\documentclass[11pt]{article}
\usepackage{graphicx}
\usepackage[a4paper, total={7in, 10.5in}]{geometry}
\usepackage{charter}
\usepackage{xcolor}

\usepackage{amsmath}
\usepackage{amsfonts}
\usepackage{amssymb}
\usepackage{amsthm}
\usepackage{hyperref}
\hypersetup{hidelinks}

\usepackage{listings}
\newcommand{\CC}{\mathbb C}
\newcommand{\NN}{\mathbb N}
\newcommand{\PP}{\mathbb P}
\newcommand{\QQ}{\mathbb Q}
\newcommand{\RR}{\mathbb R}
\newcommand{\ZZ}{\mathbb Z}

\newcommand{\MB}{\mathcal B}
\newcommand{\MC}{\mathcal C}
\newcommand{\MI}{\mathcal I}
\newcommand{\ML}{\mathcal L}
\newcommand{\MT}{\mathcal T}
\newcommand{\MR}{\mathcal R}

\newcommand{\init}{\text{in}}

\DeclareMathOperator{\Fl}{Fl}
\DeclareMathOperator{\Gr}{Gr}
\DeclareMathOperator{\Dr}{Dr}
\DeclareMathOperator{\conv}{Conv}
\DeclareMathOperator{\trop}{trop}
\DeclareMathOperator{\Trop}{Trop}

\theoremstyle{plain}
\newtheorem{theorem}{Theorem}[section]

\newtheorem{proposition}[theorem]{Proposition}
\newtheorem{conjecture}[theorem]{Conjecture}

\theoremstyle{definition}
\newtheorem{definition}[theorem]{Definition}

\theoremstyle{remark}
\newtheorem{remark}[theorem]{Remark}
\newtheorem{example}[theorem]{Example}

\title{Matching Fields in Macaulay2}
\author{Oliver Clarke}
\date{}

\begin{document}

\maketitle

\begin{abstract}
    This article introduces the package \textit{MatchingFields} for \textit{Macaulay2} and highlights some open problems.
    A matching field is a combinatorial object whose data encodes a candidate toric degeneration of a Grassmannian or partial flag variety of type $A$. Each coherent matching field is associated to a certain maximal cone of the respective tropical variety. The \textit{MatchingFields} package provides methods to construct matching fields along with their rings, ideals, polyhedra and matroids. The package also supplies methods to test whether a matching field is coherent, linkage and gives rise to a toric degeneration.
\end{abstract}

\tableofcontents

\section{Introduction}

A \textit{matching field} comes in two flavours: a Grassmannian $\Gr(k,n)$ matching field is an ordering of the elements of each $k$-subset of $[n] := \{1, \dots, n \}$ and a flag $\Fl(j_1, \dots, j_k; n)$ matching field is a set $\{L_1, \dots, L_k\}$ where $L_i$ is a Grassmannian matching field for $\Gr(j_i, n)$ for each $i \in [k]$.

Grassmannian matching fields were introduced by Sturmfels and Zelevinsky \cite{sturmfels1993maximal} to study the Newton polytope of a product of maximal minors.
In recent work, matching fields are used to parametrise a family of projective toric varieties, which can be thought of as candidates for the special fiber of a toric degeneration of a Grassmannian or flag variety. See  \cite{mohammadi2019matching, clarke2020schubert, clarke2021flag, clarke2022partial}. A matching field is said to be \textit{coherent} if it is induced by a weight matrix $w$.
In this case, the matching field is said to \textit{give rise to a toric degeneration} if the Pl\"ucker forms are a SAGBI basis for the Pl\"ucker algebra with respect to the weight order $w$. Whenever this happens, the image $\widehat w$ of $w$ under the \textit{tropical Stiefel map} \cite{fink2015stiefel} lies in the relative interior of a top-dimensional prime cone of the tropical Grassmannian \cite{speyer2004tropical} or flag variety with respect to the trivial valuation. 

The toric variety associated to a matching field is defined by its ideal,
see Definition~\ref{def: bg mf algebra ideal}, or in terms of the normal fan of the \textit{matching field polytope}, 
see Section~\ref{sec: bg polytopes}. For some families of matching fields, it is known that the matching field polytopes are related by sequences of combinatorial mutations \cite{clarke2021combinatorial, clarke2022combinatorial, clarke2022partial}. The property of a matching field giving rise to a toric degeneration has formulations in terms of the matching field ideal, and properties of its polytope that are invariant under mutation. See Propositions~\ref{prop: toric degen ideals}, \ref{prop: bg toric degen via polytope}, \ref{prop: toric degen via mf polytope volume}, and \ref{prop: toric degen via NO body}

The question of determining which matching fields give rise to toric degenerations is an open problem. For Grassmannians $\Gr(2,n)$ and $\Gr(3,m)$ with $m \in \{6,7,8\}$, it is possible to compute the tropical Grassmannian explicitly. More generally, the use of combinatorial mutations has led to the construction of families of matching fields that give rise to toric degenerations. Examples of toric degenerations also arise from representation theory. For example, the Gelfand-Tsetlin degeneration and Fang-Fourier-Littleman-Vinberg degeneration both have a description in term of matching fields \cite{clarke2022combinatorial}.

In this article, we introduce the package \textit{MatchingFields} for \textit{Macaulay2} \cite{M2}. The package facilitates working with matching fields, their ideals and polytopes, and provides methods for testing whether they are coherent and give rise to toric degenerations. Additionally, the package allows the user to construct: matroid subdivisions; algebraic matroids; and tope fields. We give examples that show how to use the package and provide exposition about techniques used to perform computations. We highlight some open problems about matching fields; for example, the matching field description of the algebraic matroid of the Grassmannian and a tope description of the free resolution of the matching field ideal. See Conjecture~\ref{conj: matching fields algebraic matroids} and Remark~\ref{rmk: free res from tope amalgamations}, respectively.

\medskip
\noindent \textbf{Overview.}
In Section~\ref{sec: background}, we fix our setup for Pl\"ucker algebras, matching field ideals, and polytopes. In Section~\ref{sec: bg plucker} we recall the Pl\"ucker embedding of type-$A$ partial flag varieties into a product of projective spaces and fix our notation for the Pl\"ucker algebra and Pl\"ucker ideal. In Section~\ref{sec: bg MF}, we recall the definition of matching field ideals and algebras. In particular, we recall what it means for a matching field to give rise to a toric degeneration of the partial flag variety. If this happens, then we say \textit{$L$ is toric}, see Definition~\ref{def: bg toric degen}. In Section~\ref{sec: bg polytopes}, we recall the definition of the matching field polytope and Newton-Okounkov body. We prove Proposition~\ref{prop: toric degen via mf polytope volume}, which shows that a coherent matching field is toric if and only if the matching field polytope has maximal volume.

In Section~\ref{sec: M2}, we introduce the package \textit{MatchingFields}.
In Section~\ref{sec: defining MFs}, we show how to construct matching fields  and view their basic properties. In particular, we define the \textit{weight matrix cone} of a matching field, which admits a test for whether a matching field is coherent, see Definition~\ref{def: weight matrix cone} and Proposition~\ref{prop: coherent iff WMC full dim}. In Section~\ref{sec: ideals and rings of MFs}, we construct the ideals and rings associated to matching fields. In particular, we explain how to check directly whether the Pl\"ucker forms are a SAGBI basis using the package \textit{SubalgebraBases} \cite{burr2023subalgebrabases}. In Section~\ref{sec: other functions}, we showcase the other functionality of the package. We explain the construction of: matching field polytopes and Newton-Okounkov bodies; matroid subdivisions of the hypersimplex induced by points in the Dressian; matching field matroids that decompose the algebraic matroid of the Grassmannian; and tope fields and their amalgamations.

\section{Background}\label{sec: background}

In this section, we recall the basic definitions and results about toric degenerations arising from matching fields. Further details can be found in \cite{clarke2022partial}. Our conventions for weighted polynomial rings are as follows. Let $K$ be a field and $Y = K[y_1, \dots, y_n]$ a polynomial ring . A \textit{weight} for $Y$ is a vector $w \in \RR^{n}$. The \textit{weight} of a monomial $cy^u \in Y$ with coefficient $c \in K \backslash \{0\}$ and exponent $u \in \ZZ^n$ is the dot product $w(c y^u) = u \cdot w$ of $u$ and $w$. The \textit{weight} $w(f)$ of a polynomial $f \in Y$ is the minimum weight of a term of $f$. The \textit{initial form} (or \textit{leading terms}) of a polynomial $f = \sum_u c_u y^u$ is the sum of minimum-weight terms of $f$:
\[
\init_w(f) = \sum_{u \cdot w \, = \, w(f)} c_u y^u.
\]
Note that the initial form of a polynomial need not be a monomial. A monomial order $\prec$ is said to \textit{refine} a weight order $w$ if for any polynomial $f$, we have $\init_\prec(f) = \init_\prec(\init_w(f))$.

\subsection{Pl\"ucker algebras}\label{sec: bg plucker}

Throughout, we fix the following setup and define the Pl\"ucker algebra for partial flag varieties of type $A$. Let $R = \CC[x_{i,j} \colon i \in [n-1], \ j \in [n]]$ be a polynomial ring whose variables are arranged into an $(n-1) \times n$ matrix $X = (x_{i,j})$.
Fix an indexing set $J = \{j_1 < \dots < j_k \} \subseteq [n-1]$. 
The partial flag variety $F = \Fl(j_1, \dots, j_k; n)$, as a set, is the collection of chains of vector subspaces of $\CC^n$:
\[
F = 
\{V_1 \subset V_2 \subset \dots \subset V_k \colon V_i \subseteq \CC^n \text{ and } \dim(V_i) = j_i \text{ for all } 1 \le i \le k\}.
\]
If $k = 1$, then $F = \Fl(j_1; n) = \Gr(j_1, n)$ is the Grassmannian of $j_1$-dimensional subspaces of $\CC^n$.
We embed $F$ into a product of projective spaces via the \textit{Pl\"ucker embedding}. 
Explicity, for each chain of vector subspaces $V = (V_1 \subset \dots \subset V_k) \in F$ we fix an $(n-1) \times n$ matrix $M_V$ such that $V_i$ is the row-span of rows $1,2, \dots, j_i$ of $M_V$. We map $M_V$ into
\[
\PP := \PP^{\binom{n}{j_1} - 1} \times \PP^{\binom{n}{j_2} - 1} \times \dots \times \PP^{\binom{n}{j_k} - 1}
\]
as follows. For each $j \in J$ and each $j$-subset $I$ of $[n]$, the $I$-th coordinate of the image of $M_V$ in the factor $\PP^{\binom{n}{j} - 1}$ of $\PP$ is the minor of $M_V$ on the columns indexed by $I$ and rows indexed by $1, 2, \dots, j$. A little linear algebra shows that the map $F \rightarrow \PP$ taking $V$ to the point in $\PP$ described above is  injective and well-defined, i.e., the map does not depend on the choice of matrices $M_V$. Therefore, the map defines the multi-projective Pl\"ucker embedding of $F$. Sometimes, it is convenient for us to consider $F$ as a projective variety. Concretely, we compose the Pl\"ucker embedding with the Segre embedding of $\PP$ into projective space. The coordinates of the embedding are the products of the transversals of coordinates of $\PP$.

The \textit{Pl\"ucker algebra} $A$ is the coordinate ring of $F$ under the Pl\"ucker embedding. Explicilty, we take $A$ to be the subalgebra of $R = \CC[x_{i,j}]$ given by
\[
A = \CC[\det(X_I) \colon I \subseteq [n], \ |I| \in J] \subseteq R
\]
where $X_I$ is the submatrix of $X$ with columns indexed by $I$ and and rows indexed by $1,2, \dots, |I|$.
We will also consider the \textit{presentation} of $A$ as the quotient $S / \MI$ where 
\[
S := \CC[P_I \colon I \subseteq [n], \ |I| \in J] \quad
\text{is a polynomial ring and}
\quad 
\MI := \ker(R \rightarrow S \colon P_I \mapsto \det(X_I))
\quad \text{is an ideal}.
\]
We refer to $\MI$ as the \textit{Pl\"ucker ideal}, which is the vanishing ideal of $F = V(\MI) \subseteq \PP$.

\subsection{Matching field ideals}\label{sec: bg MF}

A \textit{matching field} for the Grassmannian $\Gr(k,n)$ is an ordering of the elements of each $k$-subset of $[n]$. 
The ordering of a subset $\{i_1, i_2, \dots, i_k\} \subseteq [n]$ is a \textit{tuple} $(i_1, i_2, \dots, i_k)$ of the matching field. A \textit{flag matching field} for $\Fl(j_1, \dots, j_k; n)$ is a collection of matching fields $L = \{L_1, \dots, L_k\}$ where $L_i$ is a matching field for $\Gr(j_i, n)$. The set of \textit{tuples} of $L$ is the union of the set of tuples of each $L_i$.

Fix a matching field $L = \{L_1, \dots, L_k\}$ for the partial flag variety $\Fl(j_1 < \dots < j_k; n)$ and write $J = \{j_1, \dots, j_k\}$. For each tuple $(i_1, \dots, i_\ell)$ of $L$, with underlying set $I = \{i_1, \dots, i_\ell\}$,
we define the monomial $m_I = (-1)^c x_{1, i_1} x_{2, i_2} \dots x_{\ell, i_\ell} \in R$ where $c = |\{(a, b) \in [\ell] \times [\ell] \colon a < b, \ i_a > i_b \}|$ is the number of \textit{descents} of the tuple. Equivalently, the coefficient $(-1)^c$ of $m_I$ is such that $m_I$ is a term of $\det(X_I)$.

\begin{definition}\label{def: bg mf algebra ideal}
    Recall the rings $R = \CC[x_{i,j}]$ and $S = \CC[P_I]$. With the above setup, we define the \textit{monomial algebra of the matching field} 
    \[
    \CC[L] := \CC[m_I \colon I \subseteq [n], \ |I| \in J] \subseteq R.
    \]
    The \textit{matching field ideal} of $L$ is the presentation ideal of $\CC[L]$ given by  $\MI_L := \ker(S \rightarrow R \colon P_I \mapsto m_I)$.
\end{definition}

It is helpful to imagine $\CC[L]$ as a \textit{`candidate initial algebra'} of the Pl\"ucker algebra $A$. We say that a matching field $L$ is \textit{coherent} if there is a weight matrix $w \in \RR^{(n-1) \times n}$ for the polynomial ring $R$ such that $\init_w(\det(X_I)) = m_I$ for each subset $I \subseteq [n]$ with $|I| \in J$. Note, if a weight matrix $w$ exists then it uniquely identifies all tuples of the matching field $L$. In this case, we say that $L$ is the matching field \textit{induced} by $w$. 

\begin{definition}\label{def: bg toric degen}
    Let $L$ be a coherent matching field induced by a weight matrix $w$. We say that $L$ gives rise to a \textit{toric degeneration} of $F$ if the initial algebra $\init_w(A) := \CC[\init_w(f) \colon f \in A]$ of the Pl\"ucker algebra is equal to $\CC[L]$ the algebra of the matching field. For ease of notation, we say \textit{$L$ is toric} whenever $L$ gives rise to a toric degeneration of $F$. Equivalently, with the language of Remark~\ref{rmk: sagbi}, $L$ is toric if the generators of $A$ form a SAGBI basis with respect to weight order $w$.
\end{definition} 

Note that the choice of weight matrix $w$ does not affect whether $L$ is toric. That is, if $w'$ is another weight matrix that induces $L$, then we have $\init_w(A) = \CC[L] = \init_{w'}(A)$. So, the property of being toric is a well-defined property of $L$.

The property of being toric has an equivalent formulation in terms of the Pl\"ucker ideal $\MI$ and matching field ideal $\MI_L$. Given a weight matrix $w$ that induces a coherent matching field $L$, observe that $w$ is a weight for $R = \CC[x_{i,j}]$. We define the \textit{induced weight vector $\widehat w$}, for the polynomial ring $S$, by
$\widehat w(P_I) = w(m_I)$. The following is an application of \cite[Theorem~11.4]{sturmfels1996grobner}.

\begin{proposition}\label{prop: toric degen ideals}
    Let $L$ be a coherent matching field induced by a weight matrix $w$. Then $L$ is toric if and only if $\init_{\widehat w}(\MI) = \MI_L$.
\end{proposition}

\begin{example}\label{example: bg diag mf}
    The \textit{diagonal matching field} is defined so that the entries of each tuple are increasing. For instance, the diagonal matching field $L$ for $\Gr(3,6)$ has tuples: 
    \[
    (1,2,3), \, 
    (1,2,4), \,
    (1,3,4), \,
    (2,3,4), \,
    (1,2,5), \, \dots, \, (4,5,6).
    \]
    For each $3$-subset $I \subseteq [6]$, the monomial $m_I$ is the leading diagonal term of the maximal minor $\det(X_I)$:
    \[
    m_{123} = x_{1,1}x_{2,2}x_{3,3},\, 
    m_{124} = x_{1,1}x_{2,2}x_{3,4},\,
    \dots,\,
    m_{456} = x_{1,4}x_{2,5}x_{3,6}.
    \]
    In general, diagonal matching fields are coherent as they are induced by the weight matrix
    \[
    w = \begin{bmatrix}
        0 & 0 & 0 & \dots & 0 \\
        n & n-1 & n-2 & \dots & 1 \\
        2n & 2(n-1) & 2(n-2) & \dots & 2 \\
        \vdots & \vdots & \vdots & \ddots & \vdots \\
        (n-2)n & (n-2)(n-1) & (n-2)(n-2) & \dots & n-2 
    \end{bmatrix}.
    \]
    The induced weight vector for the diagonal matching field of $\Gr(3,6)$ is given by
    \[
    \widehat w_{123} = 13, \, 
    \widehat w_{124} = 11, \,
    \widehat w_{134} = 10, \, 
    \widehat w_{234} = 10, \,
    \widehat w_{125} = 9, \,
    \dots, \,
    \widehat w_{456} = 4.
    \]
    The diagonal matching field is toric for any Grassmannian and partial flag variety \cite{miller2005combinatorial}. This toric degeneration is well-studied and naturally arises from the representation theory of algebraic groups \cite{gonciulea1996degenerations, makhlin2020gelfand}. It is commonly known as the Gelfand-Tsetlin degeneration.
\end{example}

\begin{remark}
    In the \textit{MatchingFields} package, we use the characterisation in Proposition~\ref{prop: toric degen ideals} to test whether a matching field is toric. This is because \textit{Macualay2} is specialised at computing Gr\"obner bases. In particular, our implementation computes a partial Gr\"obner basis for $\init_w(\MI)$. The matching field ideal $\MI_L$ is a toric ideal so we efficiently compute it using the software package \textit{4ti2} \cite{4ti2, FourTiTwoSource}.
\end{remark}

\begin{remark}\label{rmk: sagbi}
    Given a finite set of polynomials $f_1, \dots, f_s$ of a polynomial ring equipped with a fixed term order. If the initial forms generate the initial algebra $K[\init(f_1), \dots, \init(f_s)] = \init(K[f_1, \dots, f_s])$, then $f_1, \dots, f_s$ is called a \textit{SAGBI} (Subalgebra Analogue of Gr\"obner Bases for Ideals) basis for $K[f_1, \dots, f_s]$. More generally, SAGBI bases are defined for quotients of polynomial rings \cite{stillman1999using} and finitely generated algebras equipped with discrete valuations \cite{kaveh2019khovanskii}. The name \textit{SAGBI basis} is typically used for subrings of polynomial rings or quotients of polynomial rings and \textit{Khovanskii Basis} for algebras with valuations. However, the literature is varied in its naming conventions and also includes \textit{canonical bases} and \textit{subalgebra bases}.
\end{remark}

\subsection{Matching field polytopes and Newton-Okounkov bodies}\label{sec: bg polytopes}

Fix a matching field $L = \{L_1, \dots, L_k\}$ for the partial flag variety $F = \Fl(j_1 < \dots < j_k; n)$ and let $J = \{j_1, \dots, j_k\}$. For each $i \in [k]$, define the polytope $P_i \subseteq \RR^{(n-1) \times n}$ as the convex hull of the exponent vectors of the monomials $m_I$ for each subset $I \subseteq [n]$ with $|I| = j_i$. The \textit{matching field polytope} of $L$ is the Minkowski sum $P_L = P_1 + P_2 + \dots + P_k$. Observe that $P_L$ is a lattice polytope, i.e., all its vertices lie in $\ZZ^{(n-1) \times n}$.

We recall the definition of the Ehrhart polynomial. Let $Q \subseteq \RR^d$ be a lattice polytope. The Ehrhart polynomial $E_Q(n) \in \QQ[n]$ is the polynomial such that for each $n \in \NN$, the value $E_Q(n) = |nQ \cap \ZZ^d|$ is the number of lattice points of the $n$th dilate of $Q$.

The matching field polytope gives a characterisation of toric matching fields. The result below follows directly from \cite[Theorem~1]{clarke2022partial} and \cite[Theorem~1]{clarke2022combinatorial}.

\begin{proposition}\label{prop: bg toric degen via polytope}
    Let $L$ and $L'$ be coherent matching fields for the same flag variety and assume $L$ is toric. Then $L'$ is toric if and only if the Ehrhart polynomials $E_{P_{L}}$ and $E_{P_{L'}}$ coincide.
\end{proposition}

Typically the toric matching field is taken to be the diagonal matching field. We note that the following stronger version of this result holds.

\begin{proposition}\label{prop: toric degen via mf polytope volume}
    Let $L$ be a coherent matching field for $F$. Then $L$ is toric if and only if the volume of $P_L$ is equal to the volume of the diagonal matching field  polytope (Gelfand-Tsetlin polytope) for $F$, which is maximal among all coherent matching fields for $F$.
\end{proposition}

The proof of this proposition is most easily seen from the perspective of Newton-Okounkov bodies, so we postpone its proof.

\begin{remark}
    In \cite{clarke2022combinatorial, clarke2022partial}, the proof of Proposition~\ref{prop: bg toric degen via polytope} has two parts. First, the Hilbert function of $S/\MI_L$ is equal to the Ehrhart polynomial of $P_L$. Second, the Hilbert functions of $S/\init_w(\MI)$ and $S/\MI$ are equal and, by \cite[Lemma~11.3]{sturmfels1996grobner}, we have that $\init_w(\MI) \subseteq \MI_L$. So $L$ if toric if and only if $E_{P_L}$ is equal to the Hilbert function of the coordinate ring $S/\MI$. By Proposition~\ref{prop: toric degen via mf polytope volume}, it suffices to check only the volume of the $P_L$, i.e., the leading coefficient of the Ehrhart polynomial. Moreover, only the matching fields whose polytopes have maximal volume, such as the Gelfand-Tsetlin polyotope \cite{ardila2011gelfand, liu2019gelfand}, have the toric property. 
\end{remark}

\medskip

\noindent
\textbf{Newton-Okounkov bodies.}
Fix positive integers $k$ and $n$ and write $\mathbf{0} = (0, \dots, 0) \in \RR^k$ for the all-zeros vector and $\mathbf{1} = (1, \dots, 1) \in \RR^k$ for the all-ones vector.  
Let $E \subseteq \ZZ^k \times \ZZ^n \subset \RR^k \times \RR^n$ be an affine semigroup, i.e., for all $u, v \in E$ we have $u+v \in E$, and assume that $E \cap (\mathbf{0} \times \ZZ^n) = \emptyset$. For each $u = (u_1, \dots, u_{k+n}) \in E$, we call $(u_1, \dots, u_k)$ the \textit{degree} of $u$. The \textit{Newton-Okounkov} body of $E$ is 
\[
\Delta(E) := \overline{\conv(E) \cap (\mathbf{1} \times \RR^n)},
\]
where $\overline S$ is the Euclidean closure of $S$. The Newton-Okounkov body encodes information about the limiting behaviour of $E$ \cite{kaveh2012newton}.

Consider the Pl\"ucker algebra $A \subseteq R$ for the partial flag variety $\Fl(j_1, \dots, j_k; n)$ and fix a term order $\prec$ on $R$. For each $f \in A$, we define its \textit{degree} $d(f) \in \ZZ^k$ by first defining $d(\det(X_I)) = e_i \in \ZZ^k$ the $i$th standard basis vector for each $I \subseteq [n]$ with $|I| = j_i$. The initial term $\init_\prec(f)$ is a monomial that appears in the expansion of some product of determinants $\prod_i \det(X_{I_i})$. We define $d(f) := \sum_i d(\det(X_{I_i}))$. It is straightforward to show that $d(f)$ is well-defined, i.e., it does not depend on the choice of the sets $I_i$.
The affine semigroup associated to $A$ and $\prec$ is the set of exponent vectors of initial terms of elements of $A$:
\[
E(A, \prec) := \{(d(f), e) \in \ZZ^k \times \ZZ^{(n-1) \times n} \colon f \in A \text{ and } \init_\prec(f) = x^e\}.
\]
Suppose that $L$ is a coherent matching field induced by a weight matrix $w$. Let $R$ be the ambient polynomial ring containing the Pl\"ucker algebra $A = \CC[\det(X_I)] \subseteq R$ and $\prec$ be any monomial order on $R$ that refines the weight order $w$. Let $E = E(A, \prec)$ be the affine semigroup above. Observe that the initial forms $\init_w(\det(X_I)) = \init_\prec(\det(X_I))$ generate the initial algebra $\init_\prec(A)$ if and only if the exponents of $\init_w(\det(X_I))$ are the rays that generate the cone over $E$. In other words, we have the following.

\begin{proposition}\label{prop: toric degen via NO body}
    The matching field $L$ is toric if and only if the matching field polytope $P_L$ coincides with the Newton-Okounkov body $\Delta(E)$.
\end{proposition}

We now give a proof of Proposition~\ref{prop: toric degen via mf polytope volume}.

\begin{proof}[Proof of Proposition~\ref{prop: toric degen via mf polytope volume}]
    Let $w$ be a weight vector that induces the matching field $L$ and $\prec$ be any monomial order that refines the weight order $w$. Let $E = E(A, \prec)$ be the semigroup defined above. The matching field polytope $P_L \subseteq \Delta(E)$ is a subset of the Newton-Okounkov body. The normalised volume of $\Delta(E)$ coincides with the degree of $F$ under the Pl\"ucker embedding,  hence it does not depend on the choice of $w$. So, by Proposition~\ref{prop: toric degen via NO body}, the matching field $L$ is toric if and only if $P_L$ and $\Delta(E)$ have the same volume. Since the diagonal matching field is toric, the volume of diagonal matching field polytope is equal to $\Delta(E)$. In particular, it is maximal among all polytopes of coherent matching fields.
\end{proof}

\section{Matching fields in Macaulay2}\label{sec: M2}

We introduce the package \textit{MatchingFields} for \textit{Macaulay2}. There are two main types of objects introduced by the package: \texttt{GrMatchingField} and \texttt{FlMatchingField}, which represent Grassmannian and flag matching fields respectively. The code throughout is collected in the file \texttt{matchingFieldsExampleCode.m2}, which accompanies this article.

\subsection{Constructing matching fields}\label{sec: defining MFs}

The diagonal matching field is defined with the function \texttt{diagonalMatchingField}. The tuples of a matching field are listed with the function \texttt{getTuples} and appear in reverse lexicographic order on the underlying set. For flag matching fields, the subsets are first ordered by size.

\begin{example}\label{example: diag mf construction}

    Let $D$ be the diagonal matching field for $\Gr(3,6)$
    and $D'$ be the diagonal matching field for $\Fl(1,2,3; 6)$. The tuples of $D'$ are: $(i)$ for $i \in [6]$; $(i,j)$ with $1 \le i < j \le 6$; and $(i,j,k)$ with $1 \le i < j < k \le 6$. The matching fields $D$ and $D'$ are defined and their tuples listed as follows.

\begin{lstlisting}[caption = {Diagonal matching field}]
i1 : needsPackage "MatchingFields"
o1 = MatchingFields
o1 : Package

i2 : D = diagonalMatchingField(3, 6)
o2 = Grassmannian Matching Field for Gr(3, 6)
o2 : GrMatchingField

i3 : getTuples D
o3 = {{1, 2, 3}, {1, 2, 4}, {1, 3, 4}, {2, 3, 4}, {1, 2, 5}, {1, 3, 5}, {2, 3, 5}, {1, 4, 5}, {2, 4, 5}, {3, 4, 5}, {1, 2, 6}, {1, 3, 6}, {2, 3, 6}, {1, 4, 6}, {2, 4, 6}, {3, 4, 6}, {1, 5, 6}, {2, 5, 6}, {3, 5, 6}, {4, 5, 6}}
o3 : List

i4 : D' = diagonalMatchingField({1,2,3}, 6)
o4 = Flag Matching Field for Fl(1, 2, 3; 6)
o4 : FlMatchingField

i5 : getTuples D'
o5 = {{{1}, {2}, {3}, {4}, {5}, {6}}, 
    {{1, 2}, {1, 3}, {2, 3}, {1, 4}, {2, 4}, {3, 4}, {1, 5}, {2, 5}, {3, 5}, {4, 5}, {1, 6}, {2, 6}, {3, 6}, {4, 6}, {5, 6}}, 
    {{1, 2, 3}, {1, 2, 4}, {1, 3, 4}, {2, 3, 4}, {1, 2, 5}, {1, 3, 5}, {2, 3, 5}, {1, 4, 5}, {2, 4, 5}, {3, 4, 5}, {1, 2, 6}, {1, 3, 6}, {2, 3, 6}, {1, 4, 6}, {2, 4, 6}, {3, 4, 6}, {1, 5, 6}, {2, 5, 6}, {3, 5, 6}, {4, 5, 6}}}
o5 : List
\end{lstlisting}
\end{example}

The function \texttt{matchingFieldFromPermutation} constructs a matching field $B_\sigma$, described in \cite{clarke2022partial}, for some permutation $\sigma \in S_n$. These matching fields are induced by a weight matrix that is based on the diagonal weight matrix, as in Example~\ref{example: bg diag mf}, with the entries in the second row permuted by $\sigma$. The function \texttt{getWeightMatrix} shows the weight matrix used to induce the matching field. 

\begin{example}
    Let $\sigma = (1,2,3,6,5,4)$ be a permutation. Consider the matching field $B_\sigma$ for $\Gr(3,6)$ from \cite{clarke2022partial}. The matching field is induced by the weight matrix
    \[
    M_\sigma = \begin{bmatrix}
        0 & 0 & 0 & 0 & 0 & 0 \\
        1 & 2 & 3 & 6 & 5 & 4 \\
        30 & 24 & 18 & 12 & 6  & 0
    \end{bmatrix}.
    \]
    We construct $B_\sigma$ using the package as follows.
\begin{lstlisting}[caption = {Matching field from a permutation}]
i6 : L = matchingFieldFromPermutation(3, 6, {1,2,3,6,5,4})
o6 = Grassmannian Matching Field for Gr(3, 6)
o6 : GrMatchingField

i7 : getWeightMatrix L
o7 = | 0  0  0  0  0 0 |
     | 1  2  3  6  5 4 |
     | 30 24 18 12 6 0 |
              3        6
o7 : Matrix ZZ  <--- ZZ
\end{lstlisting}

\end{example}

\begin{remark}
    The matching fields $B_\sigma$ parametrised by permutations generalise the family of \textit{block diagonal matching fields}, which were originally defined in \cite{mohammadi2019matching}. The \textit{two-block diagonal matching field $B_i$} for some $i \in [n]$ is the matching field associated to the permutation $(i, i-1, \dots, 2,1,n, n-1, \dots, i+2, i+1)$.

    Block diagonal matching fields are known to give rise to toric degenerations of: Grassmannians and their Schubert and Richardson varieties \cite{clarke2020schubert, bonala2021grRichardson} and flag varieties \cite{clarke2021flag}.
    Moreover, the polytopes of these matching fields are related by \textit{combinatorial mutations} \cite{clarke2021combinatorial, clarke2022partial}, which are certain piecewise linear maps that preserve the Ehrhart polynomial.
\end{remark}

The functions \texttt{grMatchingField} and \texttt{flMatchingField} construct a matching fields induced by a weight matrix for the Grassmannian and flag variety respectively. For the Grassmannian $\Gr(k,n)$, the parameters $k$ and $n$ are determined by the number of rows and columns of the matrix respectively. For the flag variety $\Fl(j_1, \dots, j_k; n)$, the list $j_1, \dots, j_k$ must be supplied as the first argument.

\begin{example}

Let $L_1$ be the matching field for $\Gr(2, 6)$ induced by the weight matrix $w_1$ and let $L_2$ be the matching field for $\Fl(1,2; 3)$ induced by the weight matrix $w_2$ where
\[
w_1 = 
\begin{bmatrix}
    0 & 0 & 0 & 0 & 0 & 0 \\
    2 & 4 & 1 & 3 & 6 & 5
\end{bmatrix}
\quad \text{and} \quad
w_2 = 
\begin{bmatrix}
    0 & 0 & 0 \\
    3 & 1 & 2
\end{bmatrix}.
\]
These matching fields are constructed and their tuples computed as follows.
\begin{lstlisting}[caption = {Matching fields from weight matrices}]
i8 : L1 = grMatchingField matrix {{0,0,0,0,0,0}, {2,4,1,3,6,5}}
o8 = Grassmannian Matching Field for Gr(2, 6)
o8 : GrMatchingField

i9 : getTuples L1
o9 = {{2, 1}, {1, 3}, {2, 3}, {4, 1}, {2, 4}, {4, 3}, {5, 1}, {5, 2}, {5, 3}, {5, 4}, {6, 1}, {6, 2}, {6, 3}, {6, 4}, {5, 6}}
o9 : List

i10 : getWeightMatrix L1
o10 = | 0 0 0 0 0 0 |
      | 2 4 1 3 6 5 |
              2        6
o10 : Matrix ZZ  <--- ZZ

i11 : L2 = flMatchingField({1,2}, matrix {{0,0,0}, {3,1,2}})
o11 = Flag Matching Field for Fl(1, 2; 3)
o11 : FlMatchingField

i12 : getWeightMatrix L2
o12 = | 0 0 0 |
      | 3 1 2 |
               2        3
o12 : Matrix ZZ  <--- ZZ

i13 : getTuples L2
o13 = {{{1}, {2}, {3}}, {{1, 2}, {1, 3}, {3, 2}}}
o13 : List
\end{lstlisting}
\end{example}

Matching fields are directly constructed from their tuples using the function \texttt{grMatchingField} and \texttt{flMatchingField}. The tuples may be supplied in any order. If a matching field is constructed from its tuples, then the resulting matching field may not be coherent and any subsequent functions that require a coherent matching field will produce an error. If a matching field is coherent, then a weight matrix is automatically constructed for it when required. The function \texttt{isCoherent} is used to check whether a matching field is coherent.

\begin{example}
    Let $L_3$ be the matching field for $\Gr(2,4)$ with tuples $T_3$ and $L_4$ be the matching field for $\Fl(1,2;3)$ with tuples $T_4$ where
    \[
    T_3 = \{
    (1,2),\, (1,3),\, (4,1),\, (2,3),\, (4,2),\, (3,4)\}
    \quad \text{and} \quad
    T_4 = \{
    (1),\, (2),\, (3),\,
    (1,2),\, (1,3),\, (3,2)
    \}.
    \]
    The matching field $L_3$ is not coherent. To see this, assume that a weight $w$ induces $L_3$. By adding constant vectors to each column of $w$, we do not change the induced matching field. So, we may assume that
    \[
    w = 
    \begin{bmatrix}
    0 & 0 & 0 & 0 \\
    a & b & c & d
    \end{bmatrix}
    \]
    for some $a,b,c,d \in \RR$. Since $(1,2)$ is a tuple, it follows that $a < b$. Similarly, the tuples $(2,3)$, $(3,4)$ and $(4,1)$ allow us to deduce that $a < b < c < d < a$, a contradiction. On the other hand, the matching field $L_4$ is coherent. We perform these computations and find a weight that induces $L_4$ as follows.
\begin{lstlisting}[caption = {Matching fields from tuples}]
i14 : L3 = grMatchingField(2, 4, {{1,2}, {1,3}, {4,1}, {2,3}, {4,2}, {3,4}})
o14 = Grassmannian Matching Field for Gr(2, 4)
o14 : GrMatchingField

i15 : isCoherent L3
o15 = false

i16 : getWeightMatrix L3
storicio:20:1:(3): error: expected a coherent matching field

i16 : L4 = flMatchingField({1,2}, 3, {{{1}, {2}, {3}},  {{1,2}, {1,3}, {3,2}}})
o16 = Flag Matching Field for Fl(1, 2; 3)
o16 : FlMatchingField

i17 : isCoherent L4
o17 = true

i18 : getWeightMatrix L4
o18 = | 0 0  0  |
      | 0 -2 -1 |
               2        3
o18 : Matrix ZZ  <--- ZZ
\end{lstlisting}

\end{example}

The method used for checking whether a matching field is coherent is as follows.

\begin{definition} \label{def: weight matrix cone}
Fix a matching field $L$ for the partial flag variety $\Fl(j_1, \dots, j_k; n)$. For each tuple $T = (i_1, \dots, i_s)$ of $L$ and permutation $\sigma = (\sigma_1, \dots, \sigma_s) \in Sym(T)$ of the entries of $T$, we define the half-space
\[
H(T, \sigma) := \left\{\sum_{a = 1}^s x_{a, i_a} \le \sum_{a = 1}^s x_{a, \sigma_a}  \right\}
\subseteq \RR^{(n-1) \times n}.
\]
The \textit{weight matrix cone} is the intersection of all such half spaces $\MC_L = \bigcap_{(T, \sigma)} H(T, \sigma)$.
\end{definition}

The weight matrix cone is constructed using the package with the function \texttt{weightMatrixCone} and can be used to test whether a matching field is coherent. 

\begin{proposition} \label{prop: coherent iff WMC full dim}
    Let $L$ be a matching field. The weight matrices that induce $L$ are the interior points of the weight matrix cone $\MC_L$. In particular, $L$ is coherent if and only if $\MC_L$ is full-dimensional.
\end{proposition}

The proof of this proposition follows immediately from the definitions of the weight matrix cone and of coherent matching field.



\subsection{Ideals and algebras of matching fields}\label{sec: ideals and rings of MFs}








Let $L$ be a coherent matching field. We use the function \texttt{matchingFieldIdeal} to construct the matching field ideal $\MI_L$. We require that $L$ is coherent as the ambient polynomial rings $R$ and $S$ are equipped with the weight orders $w$ and $\widehat w$ respectively, where $w$ is the weight matrix inducing $L$. The Pl\"ucker ideal $\MI$ is constructed with the function \texttt{plueckerIdeal}. To test whether a matching field is toric, we use the function \texttt{isToricDegeneration}, which checks if $\init_{\widehat w}(\MI) = \MI_L$. See Proposition~\ref{prop: toric degen ideals}.

\begin{example}\label{example: gr24 diag mf ideal}
   Let $D$ be the diagonal matching field for $\Gr(2,4)$. The Pl\"ucker ideal is a principal ideal generated by $f = P_{14} P_{23} - P_{13}P_{24} + P_{12}P_{34}$. Since $D$ is toric, the matching field ideal $\MI_D$ is generated by the initial form $\init_w(f) = P_{14}P_{23}-P_{13}P_{24}$. These ideals are constructed as follows.

\begin{lstlisting}[caption = {Matching field ideals}]
i1 : needsPackage "MatchingFields";
i2 : D = diagonalMatchingField(2, 4);
i3 : matchingFieldIdeal D
o3 = ideal(p   p    - p   p   )
            2,3 1,4    1,3 2,4

o3 : Ideal of QQ[p   ..p   , p   , p   , p   , p   ]
                  1,2   1,3   2,3   1,4   2,4   3,4

i4 : J = plueckerIdeal D
o4 = ideal(p   p    - p   p    + p   p   )
            2,3 1,4    1,3 2,4    1,2 3,4

o4 : Ideal of QQ[p   ..p   , p   , p   , p   , p   ]
                  1,2   1,3   2,3   1,4   2,4   3,4

i5 : ideal leadTerm(1, J) == matchingFieldIdeal D
o5 = true

i6 : isToricDegeneration D
o6 = true
\end{lstlisting}

\end{example}

It is possible test directly whether a matching field is toric with the \textit{SubalgebraBases} package \cite{burr2023subalgebrabases}, which allows us to compute the initial algebra of the Pl\"ucker algebra. The function \texttt{plueckerAlgebra} produces the Pl\"ucker algebra $A = \CC[\det(X_I)] \subseteq R$. We recall the function \texttt{sagbi}, from the package \textit{SubalgebraBases}, which produces an object whose generators are a (partial) SAGBI basis for the subalgebra. 

\begin{example} \label{example: gr24 subring}
    We continue with Example~\ref{example: gr24 diag mf ideal}. 
    Since $D$ is toric, the six Pl\"ucker forms $\det(X_I)$ form a SAGBI basis for the Pl\"ucker algebra $A$.

\begin{lstlisting}[caption = {Initial algebra of the Pl\"ucker algebra}]
i7 : S = plueckerAlgebra D
o7 = QQ[p_0..p_5], subring of QQ[x_(1,1)..x_(2,4)]
o7 : Subring
i8 : transpose gens S
o8 = {-2} | x_(1,1)x_(2,2)-x_(1,2)x_(2,1) |
     {-2} | x_(1,1)x_(2,3)-x_(1,3)x_(2,1) |
     {-2} | x_(1,2)x_(2,3)-x_(1,3)x_(2,2) |
     {-2} | x_(1,1)x_(2,4)-x_(1,4)x_(2,1) |
     {-2} | x_(1,2)x_(2,4)-x_(1,4)x_(2,2) |
     {-2} | x_(1,3)x_(2,4)-x_(1,4)x_(2,3) |
                            6                      1
o8 : Matrix (QQ[x   ..x   ])  <--- (QQ[x   ..x   ])
                 1,1   2,4              1,1   2,4

i9 : transpose gens sagbi S
o9 = {-2} | x_(1,1)x_(2,2)-x_(1,2)x_(2,1) |
     {-2} | x_(1,2)x_(2,3)-x_(1,3)x_(2,2) |
     {-2} | x_(1,1)x_(2,3)-x_(1,3)x_(2,1) |
     {-2} | x_(1,3)x_(2,4)-x_(1,4)x_(2,3) |
     {-2} | x_(1,2)x_(2,4)-x_(1,4)x_(2,2) |
     {-2} | x_(1,1)x_(2,4)-x_(1,4)x_(2,1) |
                            6                      1
o9 : Matrix (QQ[x   ..x   ])  <--- (QQ[x   ..x   ])
                 1,1   2,4              1,1   2,4
\end{lstlisting}

\end{example}

\begin{example}\label{example: gr36 hexagonal subring}
    Let $L$ be the matching field for $\Gr(3,6)$ induced by the weight matrix
    \[
    M = \begin{bmatrix}
        0 & 0 & 0 & 0 & 0 & 0 \\
        18 & 3 & 15 & 6 & 9 & 12 \\
        35 & 28 & 21 & 14 & 7 & 0 
    \end{bmatrix}.
    \]
    The matching field $L$ is not toric since it is an example of a \textit{hexagonal matching field} for $\Gr(3,6)$ \cite{mohammadi2019matching}. So, for any monomial order on $R = \CC[x_{i,j}]$ that refines the weight order $M$, any SAGBI basis for the Pl\"ucker algebra has more than $20$ generators. We perform these computations as follows.

\begin{lstlisting}[caption = {Initial algebra for a hexagonal matching field}]
i10 : M = matrix {{0,0,0,0,0,0},{18,3,15,6,9,12},{35,28,21,14,7,0}};
               3        6
o10 : Matrix ZZ  <--- ZZ
i11 : L = grMatchingField M;
i12 : T = plueckerAlgebra L;

i13 : numgens T
o13 = 20

i14 : numgens sagbi T
o14 = 21
\end{lstlisting}

\end{example}

\subsection{Polyhedra and other functions}\label{sec: other functions}

In this section we explain how to use the \textit{MatchingFields} package to compute: matching field polytopes and Newton-Okounkov bodies; matroidal subdivisions of hypersimplices arising from the Dressian; algebraic matroids of matching fields, which decompose the algebraic matroid of the Grassmannian; and tope fields and their amalgamations. The matching field polytopes and Newton-Okounkov bodies can be computed for both Grassmannians and flag matching fields. However, the other constructions are for Grassmannian matching fields only. In each of the following parts, we provide the necessary background and explain how to perform the computations using the package.










\medskip
\noindent \textbf{Polyhedra.}
We construct matching field polytopes and Newton-Okounkov bodies, described in Section~\ref{sec: bg polytopes}, with the functions \texttt{matchingFieldPolytope} and \texttt{NOBody} respectively. The function \texttt{NOBody} uses the \textit{SubalgebraBases} package to compute a SAGBI basis for the Pl\"ucker algebra.

\begin{example}
    Consider the hexagonal matching field $L$ induced by the weight matrix $M$ from Example~\ref{example: gr36 hexagonal subring}. Let $\prec$ be the monomial order obtained by refining the weight order $M$ by the graded reverse lexicographic order with respect to
    \[
    x_{1,1} > x_{1,2} > \dots > x_{1,6} > x_{2,1} > x_{2,2} > \dots > x_{3,6}.
    \]
    Let $E = E(A, \prec)$ be the semigroup of the initial algebra $\init_\prec(A)$. See Section~\ref{sec: bg polytopes}.
    Since $L$ is not toric, the matching field polytope $P = P_L$ is a strict subset of a Newton-Okounkov body $Q = \Delta(E)$. We compute $P$ and $Q$, their normalised volumes, and show that $P \subseteq Q$ using the package as follows.

\begin{lstlisting}[caption={Polyhedra associated $L$. The output \texttt{o4} and $\texttt{o7}$ have been trimmed for brevity. The output \texttt{o9} shows that the vertices of the Newton-Okounkov body $Q$ are the vertices of the matching field polytope $P$ together with one more vertex, which comes from a degree-$2$ generator of $\init_\prec(A)$.}]
i1 : needsPackage "MatchingFields";
i2 : L = grMatchingField matrix {{0,0,0,0,0,0},{18,3,15,6,9,12},{35,28,21,14,7,0}};

i3 : P = matchingFieldPolytope L
o3 = P
o3 : Polyhedron

i4 : vertices P
o4 = | 1 1 1 0 1 0 0 1 1 0 1 0 0 0 1 1 0 0 1 0 |
     ...
     | 0 0 0 0 0 0 0 0 0 0 1 1 1 1 1 1 1 1 1 1 |
              18        20
o4 : Matrix QQ   <--- QQ

i5 : (volume P) * (dim P)!
o5 = 38
o5 : QQ

i6 : Q = NOBody L
o6 = Q
o6 : Polyhedron

i7 : vertices Q
o7 = | 1 1 1 0 1 0 0 1 1 0 1 0 0 0 1 1 0 0 1 0 1/2 |
     ...
     | 0 0 0 0 0 0 0 0 0 0 1 1 1 1 1 1 1 1 1 1 1/2 |
              18        21
o7 : Matrix QQ   <--- QQ

i8 : (volume Q) * (dim Q)!
o8 = 42
o8 : QQ

i9 : (vertices Q)_{0 .. 19} == vertices P
o9 = true
\end{lstlisting}

\end{example}

\medskip

\noindent \textbf{Dressians.}
The Dressian $\Dr(k, n)$ \cite{herrmann2009draw} is the intersection of the tropical hypersurfaces defined by the $3$-term Pl\"ucker relations. A \textit{tropical polynomial} $\trop(f)$ is a piecewise-linear convex function defined over $\overline \RR := \RR \cup \{\infty\}$ obtained from a polynomial $f \in \RR[x_1, \dots, x_n]$ by replacing addition with minimum and multiplication with addition. So $\trop(f)$ is evaluated as the minimum of a set of linear forms. We call each such linear form a \textit{tropical monomial} of $\trop(f)$. Given a tropical polynomial $\trop(f) : {\overline{\RR}}^n \rightarrow \overline{\RR}$, its \textit{tropical hypersurface} $\MT(\trop(f)) \subseteq {\overline \RR}^n$ is the set of points where the minimum in $\trop(f)$ is attained by at least two tropical monomials. The Dressian is given explicitly by
\[
\Dr(k, n) := \bigcap_{(I, a, b, c, d)}
\MT(
\min(
P_{I \cup \{a, b\}} + P_{I \cup\{c, d \}}, \ 
P_{I \cup \{a, c\}} + P_{I \cup\{b, d \}}, \ 
P_{I \cup \{a, d\}} + P_{I \cup\{b, c \}}
))
\subseteq \overline{\RR}^{\binom nk}
\]
where the intersection is taken over all $(k-2)$-subsets $I \subseteq [n]$ and all $a < b < c < d$ in $[n] \backslash I$. On the other hand, the \textit{tropical Grassmannian} $\Trop(\Gr(k,n)) = \bigcap_{f \in \MI} \MT(\trop(f))$ is the intersection of all tropical hypersurfaces where $f$ runes over every element of the Pl\"ucker ideal $\MI$.
When $k = 2$, the Dressian coincides with the tropical Grassmannian. For $k \ge 3$ and $n \ge 6$, the Dressian strictly contains the tropical Grassmannian. 

The Dressian admits serveral combinatorial descriptions. We focus on the description in terms of \textit{matroidal subdivisions of hypersimplices}. The \textit{hypersimplex} $\Delta_{k,n} \subseteq \RR^n$ is the convex hull of the characteristic vectors of the $k$-subsets of $[n]$. Suppose that $\widehat w \in \RR^{\binom nk}$ is any weight vector. We say that the regular subdivision of $\Delta_{k,n}$ given by $\widehat w$ is \textit{matroidal} if, for each maximal cell of the subdivision, the sets indexing the vertices of the cell are the bases of a matroid. The set of weights that give matroidal subdivions of the hypersimplex $\Delta_{k,n}$ are exactly the points of the Dressian $\Dr(k,n)$ \cite[Proposition~2.2]{speyer2008tropicallinear}.

In the package \textit{MatchingFields}, given a coherent matching field $L$ induced by a weight matrix $w$, its induced weight vector is displayed with the function \texttt{getWeightPluecker}. The subsets associated to the coordinates are listed in reverse-lexicographic order, which coincides with the order of the tuples displayed with \texttt{getTuples}. See Section~\ref{sec: defining MFs} and Example~\ref{example: diag mf construction}. The matroidal subdivision obtained from the induced weight vector $\widehat w$ is computed with the function \texttt{matroidSubdivision}. The output is a list $\{\MB_1, \MB_2, \dots, \MB_s\}$ where $\MB_i$ is the list of bases for the $i$th cell of the subdivision.

\begin{example}
    Let $L$ be the matching field for $\Gr(3,5)$ induced by the weight matrix
    \[
    w = 
    \begin{bmatrix}
        0 & 0 & 0 & 0 & 0 \\
        1 & 3 & 2 & 5 & 4 \\
        10 & 0 & 20 & 40 & 30 
    \end{bmatrix}.
    \]
    The matroidal subdivision of $\Delta_{3,5}$ with respect to the induced weight vector $\widehat w$ has $3$ maximal cells, which are computed as follows.
    
\begin{lstlisting}[caption = {Matroidal subdivision of the hypersimplex $\Delta_{3,5}$}]
i10 : L = grMatchingField matrix {{0,0,0,0,0},{1,3,2,5,4},{10,0,20,40,30}};

i11 : getWeightPluecker L
o11 = {1, 1, 12, 2, 1, 12, 2, 14, 4, 24}
o11 : List

i12 : netList matroidSubdivision L
      +---------+---------+---------+---------+---------+---------+---------+---------+
o12 = |{1, 2, 3}|{1, 2, 4}|{1, 2, 5}|{2, 3, 4}|{2, 3, 5}|{1, 3, 4}|{1, 3, 5}|         |
      +---------+---------+---------+---------+---------+---------+---------+---------+
      |{1, 2, 4}|{1, 2, 5}|{2, 3, 4}|{2, 3, 5}|{2, 4, 5}|{1, 3, 4}|{1, 3, 5}|{1, 4, 5}|
      +---------+---------+---------+---------+---------+---------+---------+---------+
      |{2, 3, 4}|{2, 3, 5}|{2, 4, 5}|{1, 3, 4}|{1, 3, 5}|{1, 4, 5}|{3, 4, 5}|         |
      +---------+---------+---------+---------+---------+---------+---------+---------+
\end{lstlisting}

\end{example}

\medskip

\noindent \textbf{Algebraic matroids.} Let $V \subseteq K^n$ be an irreducible affine algebraic variety. The \textit{algebraic matroid} $M_V$ of $V$ is the matroid whose independent sets are the subsets $S \subseteq [n]$ such that the image of $V$ under the coordinate projection $\pi : K^n \rightarrow K^S$ is full-dimensional. If $V$ is not contained in any coordinate hyperplane, i.e., its ideal is monomial free, then, by \cite{yu2017algebraic}, the algebraic matroid is preserved under tropicalisation.

Consider the case of the cone over the Grassmannian $\Gr(2, n)$, which is an affine variety in $\CC^{\binom n2}$. By \cite{bernstein2017completion}, the algebraic matroid of $\Gr(2,n)$ is fully determined by the maximal cones of $\Trop(\Gr(2,n))$ whose associated metric trees \cite{speyer2004tropical} are caterpillar graphs. It is a straightforward observation that the caterpillar graph cones are exactly the cones associated to coherent matching fields.

More generally, consider the cone over the Grassmannian $\Gr(k,n)$. Fix a coherent matching field $L$ induced by a weight matrix $w$ that is toric for the Grassmannian. Let $V \subseteq \RR^{\binom nk}$ be the linear span of the cone of the Gr\"obner fan of the Pl\"ucker ideal $\MI$ containing $\widehat w$ within its relative interior. The \textit{algebraic matroid} $M_L$ of $L$ is the matroid on the ground set of $k$-subsets of $[n]$ realised by $V$. By \cite[Lemma~2]{yu2017algebraic}, it follows that each basis of $M_L$ is a basis of $M_{\Gr(k,n)}$. For the reverse direction, note that not all maximal cones of the tropical Grassmannian arise from matching fields. For example, in $\Gr(2,n)$ the non-caterpillar graphs index precisely these cones. However, for all small examples that can be currently be computed, we can verify that the algebraic matroids of matching fields are enough to construct the algebraic matroid of the Grassmannian.

\begin{conjecture}\label{conj: matching fields algebraic matroids}
    Every basis of the algebraic matroid of the Grassmannian is a basis of $M_L$ for some coherent matching field $L$, i.e.,
    \[
    \MB(M_{\Gr(k,n)}) = \bigcup_{L \text{ coherent}} \MB(M_L).
    \]
\end{conjecture}

The algebraic matroid of $L$ is computed using the function \texttt{algebraicMatroid}. The object returned by this function uses the ground set $\{0,1,\dots, \binom nk-1\}$. To view the circuits and bases of the algebraic matroid in terms of their $k$-subsets, we use the functions \texttt{algebraicMatroidCircuits} and \texttt{algebraicMatroidBases} respectively.

\begin{example}\label{example: gr26 alg matroid}
The ground set of the algebraic matroid $M_{\Gr(2,6)}$ is the edge set $E(K_6)$ of the complete graph $K_6$. We say that a cycle with labelled vertices $v_1, v_2, \dots, v_m$ is \textit{alternating} if $v_1 < v_2$, $v_2 > v_3$, $v_3 < v_4$, \dots, $v_{m-1} < v_m$, and $v_m > v_1$. The independent sets of $M_{\Gr(2,n)}$ are the subgraphs $H \subseteq K_6$ for which there exists a labelling of the vertices such that $H$ does not contain an alternating cycle \cite[Theorem~4.4]{bernstein2017completion}.

\begin{figure}
    \centering
    \includegraphics[scale=0.8]{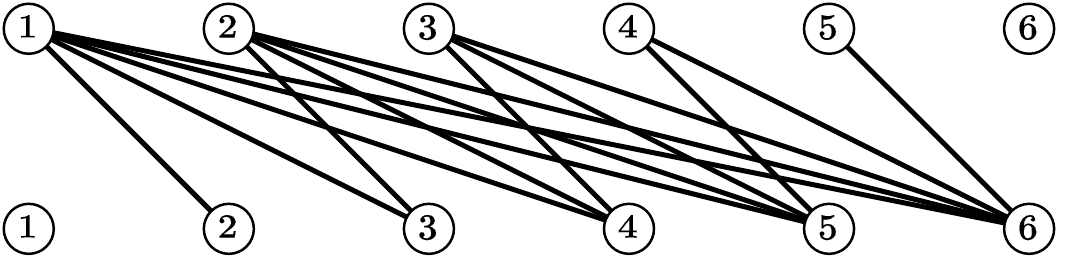}
    \caption{A graph whose graphic matroid coincides with $M_L$ from Example~\ref{example: gr26 alg matroid}. The edges are labelled with the set of incident vertices.}
    \label{fig: graph of diag mf matroid}
\end{figure}

Let $L$ be the diagonal matching field for $\Gr(2,6)$. The algebraic matroid $M_L$ is realised by the vertices of the matching field polytope $P_L$. It is straightforward to show that $M_L$ is the graphic matroid of the bipartite graph in Figure~\ref{fig: graph of diag mf matroid}. The graph has $3$ connected components and $12$ vertices, so $M_L$ has rank $9$. Its circuits are the cycles of the graph. We construct the matroid in \textit{Macaulay2}, show that it has $576$ bases, and display seven of its circuits as follows.

\begin{lstlisting}[caption = {Algebraic matroids of matching fields}]
i13 : L = diagonalMatchingField(2, 6);

i14 : algebraicMatroid L
o14 = a "matroid" of rank 9 on 15 elements
o14 : Matroid

i15 : #algebraicMatroidBases L
o15 = 576

i16 : netList (algebraicMatroidCircuits L)_{0 .. 6}
      +----------------------------------------------------+
o16 = |set {{1, 3}, {1, 4}, {2, 3}, {2, 4}}                |
      +----------------------------------------------------+
      |set {{1, 3}, {1, 5}, {2, 3}, {2, 5}}                |
      +----------------------------------------------------+
      |set {{1, 4}, {1, 5}, {2, 4}, {2, 5}}                |
      +----------------------------------------------------+
      |set {{1, 4}, {1, 5}, {3, 4}, {3, 5}}                |
      +----------------------------------------------------+
      |set {{1, 3}, {1, 5}, {2, 3}, {2, 4}, {3, 4}, {3, 5}}|
      +----------------------------------------------------+
      |set {{1, 3}, {1, 4}, {2, 3}, {2, 5}, {3, 4}, {3, 5}}|
      +----------------------------------------------------+
      |set {{2, 4}, {2, 5}, {3, 4}, {3, 5}}                |
      +----------------------------------------------------+
\end{lstlisting}

\end{example}

\medskip

\noindent
\textbf{Tope fields.}
A tope field is a generalisation of a matching field. Below we give a concise introduction to tope fields, however, a thorough exposition can be found in \cite{loho2020matching}. Our setup is modified so that it aligns with the implementation in the \textit{MatchingFields} package. A \textit{tope field} for $\Gr(k,n)$ of type $t = (t_1, \dots, t_s)$ where $k = \sum_i t_i$, is a collection of bipartite graphs called \textit{topes} on the vertices $(\ML :=[n]) \sqcup (\MR := [s])$ such that the following hold: the collection has one bipartite graph $G$ for each $k$-subset $I \subseteq [n]$; the degree vector of the vertices in $\ML$, called the \textit{left-degree vector}, is equal to the characteristic vector of $I$; and the degree vector of the vertices in $\MR$, called the \textit{right-degree vector}, is equal to the type $t$. The tope fields associated to matching fields are the tope fields of type $(1,1,\dots, 1)$. In such a case, the tuple $(i_1, \dots, i_k)$ of a matching field corresponds to the bipartite graph with edges $(i_j, j)$ for each $j \in [k] =: \MR$.

We say that a tope field is \textit{linkage} if, for each $(k+1)$-subset $S \subseteq [n]$, the bipartite graph on $L \sqcup R$ whose edges are the union of all bipartite graphs of the tope field whose left-degree vector is supported on $S$ is a forest. See \cite[Definition~3.1]{loho2020matching}.

We encode a tope field as a pair $(L, t)$ where $L$ is a Grassmannian matching field and $t = (t_1, \dots, t_s)$ is the type. Let $T = (i_{1,1}, i_{1,2}, \dots, i_{1,t_1}, i_{2,1}, i_{2,2}, \dots, i_{s, t_s})$ be a tuple of $L$. The bipartite graph corresponding to $T$ has edges $(i_{a,b}, a)$ for each $a \in [s]$ and $b \in [t_a]$. A tope field is defined from a matching field with the function \texttt{topeField}. To check that the tope field is \textit{linkage}, we use the function \texttt{isLinkage}. Given a linkage tope field $T = (L, t)$, for each $i \in [s]$, the $i$th \textit{amalagamation} of $T$ is a certain linkage tope field $(L', t + e_i)$ where $e_i$ is the $i$th standard basis vector and $L'$ is a Grassmannian matching field for $\Gr(k+1, n)$.

\begin{example}
    Let $L$ be the matching field for $\Gr(3,5)$ with tuples
    \[
    132, 142, 152, 341, 135, 145, 342, 235, 245 \text{ and } 345.
    \]
    The bipartite graphs associated to the tuples $132, 142, 341$ and $342$ are shown in Figure~\ref{fig: tope field}. This matching field is linkage so the union of these graphs is a forest. We verify this using the \textit{MatchingFields} package and compute the amalgamations as follows.

\begin{lstlisting}[caption = {Tope fields and amalgamations}]
i17 : L = grMatchingField(3, 5, {{1,3,2}, {1,4,2}, {1,5,2}, {3,4,1}, {1,3,5}, {1,4,5}, {3,4,2}, {2,3,5}, {2,4,5}, {3,4,5}});

i17 : T = topeField L
o18 = Tope field: n = 5 and type = {1, 1, 1}
o18 : TopeField

i19 : isLinkage T
o19 = true

i20 : T2 = amalgamation(2, T)
o20 = Tope field: n = 5 and type = {1, 2, 1}
o20 : TopeField

i21 : getTuples T2
o21 = {{1, 3, 4, 2}, {1, 3, 5, 2}, {1, 4, 5, 2}, {1, 3, 4, 5}, {2, 3, 4, 5}}
o21 : List

i22 : T23 = amalgamation(3, T2)
o22 = Tope field: n = 5 and type = {1, 2, 2}
o22 : TopeField

i23 : getTuples T23
o23 = {{1, 3, 4, 2, 5}}
o23 : List
\end{lstlisting}

\begin{figure}
    \centering
    \includegraphics[scale=0.8]{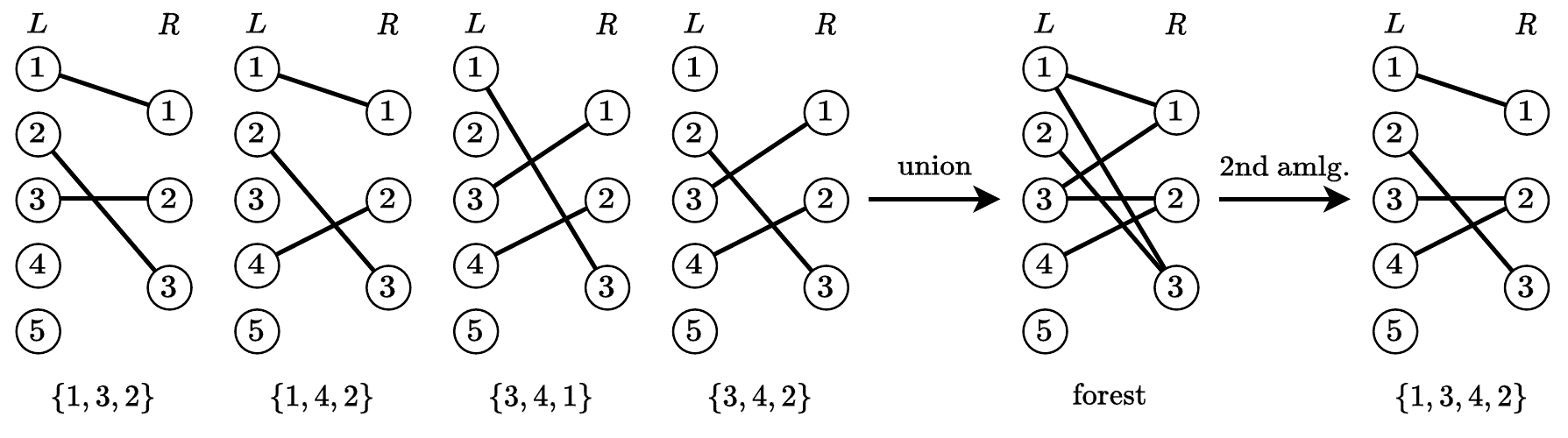}
    \caption{Left to right: four topes from $L$ whose left-degree vector is supported on $1234$; the union of their edges is a forest; the tope of the second amalgamation that is the subgraph of the forest with left-degree vector $(1,1,1,1,0)$ and right-degree vector $(1,2,1)$.}
    \label{fig: tope field}
\end{figure}

\end{example}

\begin{remark}\label{rmk: free res from tope amalgamations}
    Given a matching field $L$, it is conjectured that the collection of all sequences of amalgamations of $L$ contain the data necessary to write down the minimal free resolution of the matching field ideal $\MI_L$. In particular, the set of amalgamations are conjectured to give a combinatorial characterisation of the toric property.
\end{remark}

\bibliographystyle{plain}
\bibliography{references}

\begin{thebibliography}{10}

\bibitem{4ti2}
4ti2 team.
\newblock 4ti2---a software package for algebraic, geometric and combinatorial
  problems on linear spaces.

\bibitem{ardila2011gelfand}
Federico Ardila, Thomas Bliem, and Dido Salazar.
\newblock Gelfand--tsetlin polytopes and feigin--fourier--littelmann--vinberg
  polytopes as marked poset polytopes.
\newblock {\em Journal of Combinatorial Theory, Series A}, 118(8):2454--2462,
  2011.

\bibitem{bernstein2017completion}
Daniel~Irving Bernstein.
\newblock Completion of tree metrics and rank 2 matrices.
\newblock {\em Linear Algebra and its Applications}, 533:1--13, 2017.

\bibitem{bonala2021grRichardson}
Narasimha~Chary Bonala, Oliver Clarke, and Fatemeh Mohammadi.
\newblock Standard monomial theory and toric degenerations of {R}ichardson
  varieties in the {G}rassmannian.
\newblock {\em Journal of Algebraic Combinatorics}, 54:1159--1183, 2021.

\bibitem{burr2023subalgebrabases}
Michael Burr, Oliver Clarke, Timothy Duff, Jackson Leaman, Nathan Nichols, and
  Elise Walker.
\newblock Subalgebrabases in macaulay2.
\newblock {\em arXiv preprint arXiv:2302.12473}, 2023.

\bibitem{clarke2021combinatorial}
Oliver Clarke, Akihiro Higashitani, and Fatemeh Mohammadi.
\newblock Combinatorial mutations and block diagonal polytopes.
\newblock {\em Collectanea Mathematica}, pages 1--31, 2021.

\bibitem{clarke2022combinatorial}
Oliver Clarke, Akihiro Higashitani, and Fatemeh Mohammadi.
\newblock Combinatorial mutations of {G}elfand-{T}setlin polytopes,
  {F}eigin-{F}ourier-{L}ittelmann-{V}inberg polytopes, and block diagonal
  matching field polytopes.
\newblock {\em arXiv preprint arXiv:2208.04521}, 2022.

\bibitem{clarke2020schubert}
Oliver Clarke and Fatemeh Mohammadi.
\newblock Toric degenerations of {G}rassmannians and schubert varieties from
  matching field tableaux.
\newblock {\em Journal of Algebra}, 559:646--678, 2020.

\bibitem{clarke2021flag}
Oliver Clarke and Fatemeh Mohammadi.
\newblock Toric degenerations of flag varieties from matching field tableaux.
\newblock {\em Journal of Pure and Applied Algebra}, 225(8):106624, 2021.

\bibitem{clarke2022partial}
Oliver Clarke, Fatemeh Mohammadi, and Francesca Zaffalon.
\newblock Toric degenerations of partial flag varieties and combinatorial
  mutations of matching field polytopes.
\newblock {\em arXiv preprint arXiv:2206.13975}, 2022.

\bibitem{fink2015stiefel}
Alex Fink and Felipe Rinc{\'o}n.
\newblock Stiefel tropical linear spaces.
\newblock {\em Journal of Combinatorial Theory, Series A}, 135:291--331, 2015.

\bibitem{gonciulea1996degenerations}
Nicolae Gonciulea and Venkatramani Lakshmibai.
\newblock Degenerations of flag and schubert varieties to toric varieties.
\newblock {\em Transformation Groups}, 1:215--248, 1996.

\bibitem{M2}
Daniel~R. Grayson and Michael~E. Stillman.
\newblock Macaulay2, a software system for research in algebraic geometry.
\newblock Available at \url{http://www.math.uiuc.edu/Macaulay2/}.

\bibitem{herrmann2009draw}
Sven Herrmann, Anders Jensen, Michael Joswig, and Bernd Sturmfels.
\newblock How to draw tropical planes.
\newblock {\em the electronic journal of combinatorics}, 16(2):R6, 2009.

\bibitem{kaveh2012newton}
Kiumars Kaveh and Askold~G Khovanskii.
\newblock Newton-okounkov bodies, semigroups of integral points, graded
  algebras and intersection theory.
\newblock {\em Annals of Mathematics}, pages 925--978, 2012.

\bibitem{kaveh2019khovanskii}
Kiumars Kaveh and Christopher Manon.
\newblock Khovanskii bases, higher rank valuations, and tropical geometry.
\newblock {\em SIAM Journal on Applied Algebra and Geometry}, 3(2):292--336,
  2019.

\bibitem{liu2019gelfand}
Ricky~I Liu, Karola M{\'e}sz{\'a}ros, and Avery~St Dizier.
\newblock Gelfand--tsetlin polytopes: A story of flow and order polytopes.
\newblock {\em SIAM Journal on Discrete Mathematics}, 33(4):2394--2415, 2019.

\bibitem{loho2020matching}
Georg Loho and Ben Smith.
\newblock Matching fields and lattice points of simplices.
\newblock {\em Advances in Mathematics}, 370:107232, 2020.

\bibitem{makhlin2020gelfand}
Igor Makhlin.
\newblock Gelfand--tsetlin degenerations of representations and flag varieties.
\newblock {\em Transformation Groups}, pages 1--34, 2020.

\bibitem{miller2005combinatorial}
Ezra Miller and Bernd Sturmfels.
\newblock {\em Combinatorial commutative algebra}, volume 227.
\newblock Springer, 2005.

\bibitem{mohammadi2019matching}
Fatemeh Mohammadi and Kristin Shaw.
\newblock Toric degenerations of {Grassmannians} from matching fields.
\newblock {\em Algebraic Combinatorics}, 2(6):1109--1124, 2019.

\bibitem{speyer2004tropical}
David Speyer and Bernd Sturmfels.
\newblock The tropical {G}rassmannian.
\newblock {\em Advances in Geometry}, 4(3):389--411, 2004.

\bibitem{speyer2008tropicallinear}
David~E Speyer.
\newblock Tropical linear spaces.
\newblock {\em SIAM Journal on Discrete Mathematics}, 22(4):1527--1558, 2008.

\bibitem{stillman1999using}
Michael Stillman and Harrison Tsai.
\newblock Using {SAGBI} bases to compute invariants.
\newblock {\em Journal of Pure and Applied Algebra}, 139(1-3):285--302, 1999.

\bibitem{FourTiTwoSource}
Mike Stillman, Josephine Yu, and Sonja Petrovic.
\newblock {FourTiTwo: Interface to 4ti2. Version~1.0}.
\newblock A \emph{Macaulay2} package available at
  \url{https://github.com/Macaulay2/M2/tree/master/M2/Macaulay2/packages}.

\bibitem{sturmfels1993maximal}
B.~Sturmfels and A.~Zelevinsky.
\newblock Maximal minors and their leading terms.
\newblock {\em Advances in Mathematics}, 98(1):65--112, 1993.

\bibitem{sturmfels1996grobner}
Bernd Sturmfels.
\newblock {\em Grobner bases and convex polytopes}, volume~8.
\newblock American Mathematical Soc., 1996.

\bibitem{yu2017algebraic}
Josephine Yu.
\newblock Algebraic matroids and set-theoretic realizability of tropical
  varieties.
\newblock {\em Journal of Combinatorial Theory, Series A}, 147:41--45, 2017.

\end{thebibliography}

\bigskip

\noindent 
{\sc Oliver Clarke} \newline
{\it Email:} oliver.clarke@ed.ac.uk \newline
{\it Address:} School of Mathematics, University of Edinburgh, United Kingdom

\end{document}